\documentclass[11pt,reqno]{amsart}%
\usepackage{graphicx}
\usepackage{amsmath}
\usepackage{amsfonts}
\usepackage{amssymb}
\usepackage{amsthm}
\usepackage{enumerate}
\usepackage{color}
\usepackage{cite}
\usepackage{url}
\usepackage[margin=1in]{geometry}%
\providecommand{\U}[1]{\protect\rule{.1in}{.1in}}
\newtheorem{theorem}{Theorem}[section]
\newtheorem{proposition}[theorem]{Proposition}
\newtheorem{lemma}[theorem]{Lemma}

\theoremstyle{definition}
\newtheorem{definition}[theorem]{Definition}
\newtheorem{example}[theorem]{Example}
\newtheorem{conjecture}[theorem]{Conjecture}

\begin{document}
\title{Multiplicities of eigenvalues of tensors}
\author{Shenglong Hu}
\address{Department of  Mathematics, School of Science, Tianjin University, Tianjin 300072, China.}
\email{timhu@tju.edu.cn}
\address{Department of  Mathematics, National University of Singapore, 10 Lower Kent Ridge Road, 119076, Singapore.}
\email{mathush@nus.edu.sg}

\author{Ke Ye}
\address{Department of Mathematics, University of Chicago, Chicago, IL 60637-1514, USA.}
\email{kye@math.uchicago.edu}

\begin{abstract}
We study in this article multiplicities of eigenvalues of tensors. There are two natural multiplicities associated to an eigenvalue $\lambda$ of a tensor: algebraic multiplicity $\operatorname{am}(\lambda)$ and geometric multiplicity $\operatorname{gm}(\lambda)$. The former is the multiplicity of the eigenvalue as a root of the characteristic polynomial, and the latter is the dimension of the eigenvariety (i.e., the set of eigenvectors) corresponding to the eigenvalue.

We show that the algebraic multiplicity could change along the orbit of tensors by the orthogonal linear group action, while the geometric multiplicity of the zero eigenvalue is invariant under this action, which is the main difficulty to study their relationships. However, we show that for a generic tensor, every eigenvalue has a unique (up to scaling) eigenvector, and both the algebraic multiplicity and geometric multiplicity are one.
In general, we suggest for an $m$-th order $n$-dimensional tensor the relationship
\[
\operatorname{am}(\lambda)\geq  \operatorname{gm}(\lambda)(m-1)^{\operatorname{gm}(\lambda)-1}.
\]
We show that it is true for serveral cases, especially when the eigenvariety contains a linear subspace of dimension $\operatorname{gm}(\lambda)$ in coordinate form. As both multiplicities are invariants under the orthogonal linear group action in the matrix counterpart, this generalizes the classical result for a matrix: the algebraic mutliplicity is not smaller than the geometric multiplicity.
\end{abstract}
\maketitle

%%%%%%%%%%%%%%%%%%%%%%%%%%%%%%%%%%%%%%%%%%%%%%%%%%%%%%%%%%%%%%%%%%%%%%%%%%%%%%%%%%%%%%%%%%%
\section{Introduction}\label{sec:intro}
Tensors are ubiquitous and inevitable generalizations of matrices.  Eigenvalue of a tensor, as a natural generalized notion of the eigenvalue of a square matrix, is proposed in recent years to study tensors independently by Lim \cite{l05} and Qi \cite{q05}. Among others, the number of eigenvalues \cite{q05,cs13,fo14,nqww07,lqz13}, Perron-Frobenius theorem for nonnegative tensors \cite{cpz08}, and applications to spectral hypergraph theory \cite{cd12,hq12,hq13b,hq14,l07,q14} are the well-studied topics. 

The eigenvalues of  a tensor are the roots of the characteristic polynomial, which is a monic polynomial with degree being determined by the order and the dimension of the tensor \cite{q05,hhlq13}. Therefore, for an eigenvalue, we can define its algebraic multiplicity as the multiplicity of the eigenvalue as a root of the characteristic polynomial, and the geometric multiplicity as the dimension of the set of eigenvectors corresponding to this eigenvalue. However, the set of eigenvectors of an eigenvalue is not a linear subspace in general, which is due to the nonlinearity of the eigenvalue equations of tensors. The two multiplicities are generalizations of those for eigenvalues of matrices. 
In the matrix counterpart, the classical linear algebra says that the algebraic multiplicity always greater than the geometric multiplicity. While, from the literature of eigenvalues of tensors, it is not clear what is the relationship between the algebraic multiplicity and the geometric multiplicity of an eigenvalue. It can also been seen from applications \cite{fo14,cd12,hq14,oo12} that the understanding of the multiplicities of an eigenvalue is a necessity. 

This article is devoted to the study on this topic, the rest of which is organized as follows. Basic notions, such as determinants, eigenvalues and their multiplicities of tensors, are presented in the next section. 
We discuss the multiplicities of the zero eigenvalue of a tensor in Sections~\ref{sec:zero} and \ref{sec:lower}. More precisely, we show that the algebraic multiplicity is not an invariant of a tensor under the orthogonal linear group action, whereas the geometric multiplicity is in Section~\ref{sec:zero}. Therefore, the general relation between the two multiplicities would be a subtle problem to uncover and we make a reasonable conjecture on it in this section. 

%%%%%%%%%
\subsection*{Conjecture}
Suppose that an $m$-th order $n$-dimensional tensor $\mathcal A$ has an eigenvalue $\lambda$ with the set of eigenvectors (i.e., eigenvariety) $V(\lambda)$ possessing $\kappa$ irreducible components $V_1,\dots, V_\kappa$. Then, the conjecture is that 
\[
\operatorname{am}(\lambda)\geq \sum_{i=1}^\kappa\operatorname{dim}(V_i)(m-1)^{\operatorname{dim}(V_i)-1},
\]
where $\operatorname{am}(\lambda)$ is the algebraic multiplicity of $\lambda$ and $\operatorname{dim}(V_i)$ is the dimension of $V_i$ as a subvariety in the affine space. Since $\operatorname{dim}(V_i)=\operatorname{gm}(\lambda)$, the geometric multiplicity, for some $i$, this general conjecture implies the specific conjecture between the two multiplicities: 
\begin{equation}\label{intro:eq}
\operatorname{am}(\lambda)\geq \operatorname{gm}(\lambda)(m-1)^{\operatorname{gm}(\lambda)-1}.
\end{equation} 

The relation \eqref{intro:eq} is a generalization of that in the matrix counterpart (i.e., $m=2$): the algebraic multiplicity is greater than the geometric multiplicity, which can be proved by the basic case when the eigenspace is a linear subspace in coordinate form and the fact that in this situation both multiplicities are invariants under the orthogonal linear group action (cf.\ Section~\ref{sec:matrix}).  

The conjecture \eqref{intro:eq} is shown to be true for several interesting cases in the sequel.
Section~\ref{sec:lower} is for lower rank symmetric tensors which include orthogonal decomposable symmetric tensors, and Section~\ref{sec:linear} is for tensors whose eigenvarieties contain linear subspaces of dimension $\operatorname{gm}(\lambda)$ in coordinate form, which generalizes the relation in the matrix counterpart.
In Section~\ref{sec:generic}, we discuss the generic case. We show that for a generic tensor, both the algebraic multiplicity and geometric multiplicity are one. Likewise, in this case, we have the relation \eqref{intro:eq}. Moreover, we use the Shape Lemma to show that in this case, each eigenvalue has a unique (up to scaling) eigenvector. In Section~\ref{sec:duality}, we discuss symmetric tensors, for which the determinantal hypersurface is the dual of the Veronese variety. We study the eigenvectors from the perspective of variety duality. Some final remarks are given in the last section. 

%%%%%%%%%%%%%%%%%%%%%%%%%%%%%%%%%%%%%%%%%%%%%%%%%%%%%%%%%%%%%%%%%%%%%%%%%%%%%%%%%%%%%%%%%%%
\section{Preliminaries}\label{sec:pre}

%----------------------------------------------------------------Definition
\begin{definition}[Classical Resultant]\label{def:resultant}
For fixed positive degree $d$, let $f_i:=\sum_{|\alpha|=d}c_{i,\alpha}\mathbf{x}^\alpha$ be a homogenous polynomial of degree $d$ in $\mathbb{C}[\mathbf{x}]$ for $i\in\{1,\ldots,n\}$. Then the unique polynomial $\mbox{RES}\in\mathbb{Z}[\{u_{i,\alpha}\}]$, which has the following properties, is called the {\rm resultant} of degrees $(d,\ldots, d)$.
\begin{itemize}
\item [(i)] The system of polynomial equations $f_1=\cdots=f_n=0$ has a nontrivial solution in $\mathbb{C}^n$ if and only if $\mbox{Res}(f_1,\ldots,f_n)=0$.
\item [(ii)] $\mbox{Res}(x_1^{d_1},\ldots,x_n^{d_n})=1$.
\item [(iii)] $\mbox{RES}$ is an irreducible polynomial in $\mathbb{C}[\{u_{i,\alpha}\}]$.
\end{itemize}
\end{definition}

The differences between the capital notation $\mbox{RES}$ and the notation $\mbox{Res}(f_1,\ldots,f_n)$ for a specific system $(f_1,\ldots,f_n)$ are: the former is understood as a polynomial in the variables $\{u_{i,\alpha}\;|\;|\alpha|=d,\;i\in\{1,\ldots,n\}\}$ and the latter is understood as the evaluation of $\mbox{RES}$ at the point $\{u_{i,\alpha}=c_{i,\alpha}\}$ with $\{c_{i,\alpha}\}$ being given by $f_i$. Therefore, $\mbox{Res}(f_1,\ldots,f_n)$ is a number in $\mathbb{C}$.

Let $\mathcal T=(t_{i_1\ldots
i_m})$ be an $m$-th order $n$-dimensional tensor,
$\mathbf{x}=(x_i)\in\mathbb{C}^n$ (the $n$-dimensional complex
space) and $\mathcal T\mathbf{x}^{m-1}$ be an $n$-dimensional vector
with its $i$-th element being
\[
\sum_{i_2=1}^n\cdots\sum_{i_m=1}^nt_{ii_2\ldots i_m}x_{i_2}\cdots
x_{i_m}.
\]
It is actually a tensor contraction. The set of all $m$-th order $n$-dimensional tensors is denoted as $\mathbb T(\mathbb C^n,m)$, which is $\otimes^m\mathbb C^n=\mathbb C^n\otimes\dots\otimes\mathbb C^n$ ($m$ times). 

%----------------------------------------------------------------Definition
\begin{definition}[Tensor Determinant]\label{def:determinant}
Let $\mbox{RES}$ be the resultant of degrees $(m-1,\ldots, m-1)$ which is a polynomial in variables $\{u_{i,\alpha}\;|\;|\alpha|=m-1,\;i\in\{1,\ldots,n\}\}$. 
The determinant $\mbox{DET}$ of $m$-th order $n$-dimensional tensors is defined as the polynomial with variables $\{v_{ii_2\ldots i_m}\;|\;i,i_2,\ldots,i_m\in\{1,\ldots,n\}\}$ through replacing $u_{i,\alpha}$ in the polynomial $\mbox{RES}$ by $\sum_{(i_2,\ldots,i_m)\in \mathbb{X}(\alpha)}v_{ii_2\ldots i_m}$. Here $\mathbb{X}(\alpha):=\{(i_2,\ldots,i_m)\in\{1,\ldots,n\}^{m-1}\;|\;x_{i_2}\cdots x_{i_m}=\mathbf{x}^{\alpha}\}$. The value of the determinant $\mbox{Det}(\mathcal T)$ of a specific tensor $\mathcal T$ is defined as
the evaluation of $\mbox{DET}$ at the point $\{v_{ii_2\ldots i_m}=t_{ii_2\ldots i_m}\}$.
\end{definition}
For the convenience of the subsequent analysis, we define $\mbox{DET}(\mathcal T)$ as the polynomial with variables $\{t_{ii_2\ldots i_m}\;|\;i,i_2,\ldots,i_m\in\{1,\ldots,n\}\}$ through replacing $v_{ii_2\ldots i_m}$ in $\mbox{DET}$ by $t_{ii_2\ldots i_m}$. There can be some specific relations on the variables $\{t_{ii_2\ldots i_m}\}$, such as some being zero. In this case, $\mathcal T$ is considered as a tensor of indeterminate variables, while it is considered as a tensor of numbers in $\mathbb C$ when we talk about $\mbox{Det}(\mathcal T)$.

%------------------------
\begin{definition}[Eigenvalues and Eigenvectors]\label{def:eigenvalue-eigenvector}
Let tensor $\mathcal T=(t_{ii_2\ldots i_m})\in\mathbb{T}(\mathbb{C}^n,m)$. A number $\lambda\in\mathbb C$ is called an eigenvalue of $\mathcal T$, if there exists a vector $\mathbf x\in\mathbb C^n\setminus\{\mathbf 0\}$ which is called an eigenvector such that
\[
\mathcal T\mathbf x^{m-1}=\lambda\mathbf x^{[m-1]},
\]
where $\mathbf x^{[m-1]}$ is an $n$-dimensional vector with its $i$-th component being $x_i^{m-1}$. 
\end{definition}

The next proposition is a direct consequence of the definition, see also \cite{hhlq13}.
The set of eigenvalues of a given tensor $\mathcal T$ is always finite \cite{hhlq13}, which is denoted as $\sigma(\mathcal T)$, and it is the set of roots of the univariate polynomial 
\[
\chi(\lambda)=\operatorname{Det}(\lambda\mathcal I-\mathcal T)
\]
which is called the \textit{characteristic polynomial} of $\mathcal T$. 

%--------------
\begin{proposition}\label{prop:irreducible}
Let the tensor space being $\mathbb T(\mathbb C^n,m)$.
The determinant $\operatorname{DET}\in\mathbb C[\{v_{i_1\dots i_m}\}]$ is irreducible, and
for any $\mathcal T\in\mathbb T(\mathbb C^n,m)$
\[
\operatorname{Det}(\mathcal T)=0\ \iff\ 0\in\sigma(\mathcal T).
\]
\end{proposition}

For a given eigenvalue $\lambda\in\sigma(\mathcal T)$, the set of eigenvectors corresponding to the eigenvalue $\lambda$ is denoted as $V_{\mathcal T}(\lambda)$ (simplified as $V(\lambda)$ whenever it is clear from the content) \footnote{Here we inlude the vector $\mathbf 0$ for convenience.}:
\[
V_{\mathcal T}(\lambda):=\big\{\mathbf x\in\mathbb C^n: \mathcal T\mathbf x^{m-1}=\lambda\mathbf x^{[m-1]}\big\}.
\]
One can either treat $V(\lambda)$ as a variety in $\mathbb {PC}^{n-1}$ (the projective space of $\mathbb C^n$) or as a variety in $\mathbb C^n$. In this article, we will take $V(\lambda)$ as an affine variety. Instead of eigenspace as in linear algebra, we use the nomenclature \textit{eigenvariety} for $V(\lambda)$, which can be sensed from the next example. 
%-------------
\begin{example}
Take the tensor $\mathcal T\in\mathbb T(\mathbb C^2,3)$ as
\[
t_{111}=2, \ t_{122}=1,\ t_{222}=1,\ \text{and }t_{ijk}=0\ \text{for the others}. 
\]
The eigenvalue equations are
\[
2x_1^2+x_2^2=\lambda x_1^2,\ x_2^2=\lambda x_2^2. 
\]
Obviously, $1\in\sigma(\mathcal T)$ is an eigenvalue of $\mathcal T$, and
\[
V(1)=\big\{\alpha(\sqrt{-1},1),\beta(-\sqrt{-1},1), \ \alpha, \beta\in \mathbb C\big\}=\mathbb C (\sqrt{-1},1)\cup \mathbb C (-\sqrt{-1},1).
\]
It is easy to see that $V(1)$ is reducible and hence not a linear subspace, and the dimension is $1$. 
\end{example}

%------------------------
\begin{definition}[Multiplicity]\label{def:mutliplicity}
Let tensor $\mathcal T=(t_{ii_2\ldots i_m})\in\mathbb{T}(\mathbb{C}^n,m)$. The \textit{algebraic multiplicity} $\operatorname{am}(\lambda)$ of an eigenvalue $\lambda\in\sigma(\mathcal T)$ is the multiplicity of $\lambda$ as a root of the characteristic polynomial $\chi(\lambda)$. The \textit{geometric multiplicity} $\operatorname{gm}(\lambda)$ of an eigenvalue $\lambda\in\sigma(\mathcal T)$ is  the dimension of the variety $V(\lambda)\subseteq\mathbb{C}^{n}$.
\end{definition}

We refer to \cite{clo98,gkz94,h77} for the basic algebraic geometry terminologies. The definitions of multiplicities boils down to the classical ones for matrices \cite{hj85}. 

%-----------------
\begin{proposition}[Disjoint]\label{prop:disoint}
Let tensor $\mathcal T=(t_{ii_2\ldots i_m})\in\mathbb{T}(\mathbb{C}^n,m)$. For any $\lambda,\mu\in\sigma(\mathcal T)$, we have
\[
V(\lambda)\cap V(\mu)=\{\mathbf 0\} \ \iff\ \lambda\neq \mu,
\]
and
\[
V(\lambda)=V(\mu) \ \iff\ \lambda=\mu.
\]
\end{proposition}

\begin{proof}
We only need to show that $V(\lambda)\cap V(\mu)=\{\mathbf 0\}$ whenever $\lambda\neq \mu$. 
Suppose on the contrary that $(\lambda, \mathbf x)$ and $(\mu,\mathbf x)$ are two eigenpairs of $\mathcal T$ with $\lambda\neq\mu$ and $\mathbf x\neq\mathbf 0$. Then, we have that 
\[
\lambda\mathbf x^{[m-1]}=\mathcal T\mathbf x^{m-1}=\mu\mathbf x^{[m-1]}
\]
which is a contradiction, since $\mathbf x\neq\mathbf 0$. 
\end{proof}

%%%%%%%%%%%%%%%%%%%%%%%%%%
\section{The Multiplicities of the Zero Eigenvalue}\label{sec:zero}
In this section, we discuss some properties of the multiplicities of the zero eigenvalue of a tensor. The general purpose is to have an intuition on the difficulty and the expected relationships between the two multiplicities. It begins with the matrix counterpart. 
%---------------------------------------------------------
\subsection{The Matrix Case}\label{sec:matrix}
There would be several ways to prove that the algebraic multiplicity of an eigenvalue of a matrix is no smaller than its geometric multiplicity, one standard approach would be the following. Let $A\in\mathbb C^{n\times n}$ and $\lambda_*\in\sigma(A)$ be an eigenvalue of $A$ with the subspace of eigenvectors being
\begin{equation}\label{matrix:coord}
V(\lambda_*)=\{\mathbf x\in\mathbb C^n : x_{k+1}=\dots=x_n=0\}. 
\end{equation}
So, $\operatorname{gm}(\lambda_*)=k$. We can partition $A$ according to $\{1,\dots,n\}=\{1,\dots,k\}\cup\{k+1,\dots,n\}$ as
\[
A=\begin{bmatrix}A_1&A_2\\A_3&A_4\end{bmatrix}
\]
with $A_1\in\mathbb C^{k\times k}$ and $A_4\in\mathbb C^{(n-k)\times (n-k)}$. Since the eigenspace of $A$ for $\lambda_*$ has the form \eqref{matrix:coord}, we can get that
\[
A_1=\lambda_* I\ \text{and }A_3=\mathbf 0.
\]
Therefore,
\[
\chi(\lambda)=\operatorname{Det}(\lambda I-A)=\operatorname{Det}(\lambda I-A_1)\operatorname{Det}(\lambda I-A_4)=(\lambda-\lambda_*)^k\operatorname{Det}(\lambda I-A_4). 
\]
Henceforth, 
\begin{equation}\label{matrix:rel}
\operatorname{am}(\lambda_*)\geq k=\operatorname{gm}(\lambda_*).
\end{equation}

For the general case when $V(\lambda_*)$ is not in the coordinate form \eqref{matrix:coord}, we can always adopt an orthogonal linear transformation $P\in\mathbb{O}(n,\mathbb C)$ (the orthogonal linear group of order $n$ over the field $\mathbb C$) such that
\[
P V(\lambda_*):=\{P\mathbf x : \mathbf x\in V(\lambda_*)\}=\{\mathbf x\in\mathbb C^n : x_{k+1}=\dots=x_n=0\}. 
\]
It is a matter of fact that the matrix 
\[
B=PAP^\mathsf{T}
\]
has $\lambda_*$ being an eigenvalue with the corresponding eigenspace being $PV(\lambda_*)$ and hence in the coordinate form \eqref{matrix:coord}. So, the proceeding discussion implies that 
\[
\operatorname{Det}(\lambda I-B)=(\lambda-\lambda_*)^kp(\lambda)
\]
for some monic polynomial $p\in\mathbb C[\lambda]$ of degree $n-k$. While,
\[
\operatorname{Det}(\lambda I-B)=\operatorname{Det}(\lambda I-PAP^\mathsf{T})=\operatorname{Det}(P(\lambda I-A)P^\mathsf{T})=\operatorname{Det}(\lambda I-A). 
\]
Henceforth, \eqref{matrix:rel} still holds for the general case. 

In summary, the above technique relies heavily on the fact that the eigenvalues (or equivalently the zero eigenvalue) together with their multiplicities of a matrix are invariants under the orthogonal group action. 
Likewise, we can associate a group action to tensors like that for matrices. While, the eigenvalues of a tensor are not invariants under this action anymore \cite{q05}. However, the zero eigenvalue is an invariant of a tensor under this action \cite{hhlq13}. 

%%%%%%%%%%%%%%%%%%
\subsection{The Zero Eigenvalue}\label{sec:tensor-zero}
Given a tensor $\mathcal A\in \mathbb{T}(\mathbb{C}^n,m)$, and matrices $P^{(i)}\in\mathbb C^{r\times n}$ for $i=1,\dots,m$, we can define the matrix-tensor multiplication $(P^{(1)},P^{(2)}\dots,P^{(m)})\cdot\mathcal A\in\mathbb T(\mathbb C^r,m)$ as  (cf.\ \cite{l07,l12})
\begin{multline*}
\big[(P^{(1)},P^{(2)}\dots,P^{(m)})\cdot\mathcal A\big]_{i_1i_2\dots i_m}:=\\
\sum_{j_1,j_2,\dots,j_m=1}^na_{j_1j_2\dots j_m}p^{(1)}_{i_1j_1}p^{(2)}_{i_2j_2}\dots p^{(m)}_{i_mj_m}\ \text{for all }i_1,i_2\dots,i_m=1,\dots,r. 
\end{multline*}
This generalizes the matrix multiplication: it is easy to see that for a matrix $A$, we have $(P,P)\cdot A=PAP^\mathsf{T}$. In general, when
\[
P^{(1)}=\cdots=P^{(m)}=P,
\]
we simplify $(P,P\dots,P)\cdot\mathcal A$ as $P\cdot\mathcal A$. 
It is a direct calculation to see that for all $P\in\mathbb C^{r\times n}$, 
\begin{equation}\label{group-equ}
\big(P\cdot\mathcal A\big)\mathbf x^{m-1}=P\bigg\{\big[(I,P,\dots,P)\cdot\mathcal A\big]\mathbf x^{m-1}\bigg\}=P\big[\mathcal A\big(P^\mathsf{T}\mathbf x)^{m-1}\big)\big]\ \text{for all }\mathbf x\in\mathbb C^r,
\end{equation}
where $I$ is the identity matrix in $\mathbb C^{n\times n}$. 

Particularly, when $r=n$, and $P\in\mathbb{GL}(n,\mathbb C)$, it is easy to see that the system of polynomial equations
\[
\mathcal A\mathbf x^{m-1}=\mathbf 0
\]
becomes
\[
\big[(I,P^\mathsf{-T},\dots,P^\mathsf{-T})\cdot\mathcal A\big]\mathbf y^{m-1}=\mathbf 0
\]
under the coordinate change $\mathbb C^n\to\mathbb C^n$ through $\mathbf x\mapsto \mathbf y=P\mathbf x$. 

When $P\in\mathbb{O}(n,\mathbb C)$, the matrix-tensor multiplication $P\cdot\mathcal A$ becomes a natural orthogonal linear group action on the space $\mathbb{T}(\mathbb{C}^n,m)$. We have also the group action on $\mathbb{GL}(n,\mathbb C)$, the general linear group. It is a matter of fact that the determinantal hypersurface $\mathbb V(\operatorname{DET})$ is invariant under this group action (see \cite{gkz94} or directly from Definition~\ref{def:determinant}). The group action structure of the matrix-tensor multiplication for $\mathbb{O}(n,\mathbb C)$ will only be encountered in this section, whereas in the others solely being an algebraic operation between matrices and tensors is understood. 

The next result shows the subtlety of algebraic multiplicity from the geometric prespective. 
%---------------------------
\begin{proposition}[Zero Algebraic Multiplicity]\label{prop:noninvariant}
The algebraic multiplicity of the zero eigenvalue of a tensor is not invariant under the group action by $\mathbb{O}(n,\mathbb C)$ as above.
\end{proposition}

\begin{proof}
We show by an example that the codegree one coefficient of the characteristic polynomial 
\[
\operatorname{Det}(\lambda\mathcal I-\mathcal A)
\]
of a tensor $\mathcal A$
is zero, while it is nonzero for that of $P\cdot\mathcal A$
\[
\operatorname{Det}(\lambda\mathcal I-P\cdot\mathcal A)
\]
for some $P\in\mathbb{O}(n,\mathbb C)$. 

Let $\mathcal A\in\mathbb T(\mathbb C^2,3)$ be given as
\[
a_{112}=1\ \text{and }a_{ijk}=0\ \text{for the other }i,j,k\in\{1,2\}.
\]
$\mathcal A$ is a tensor in the upper triangular form and by \cite[Proposition~5.1]{hhlq13} we have that
\[
\operatorname{Det}(\lambda\mathcal I-\mathcal A)=(\lambda-a_{111})^2(\lambda-a_{222})^2=\lambda^4.
\]
Therefore, the algebraic multiplicity of the zero eigenvalue of $\mathcal A$ is $4$, and zero is the unique eigenvalue of $\mathcal A$. 

Let $P\in\mathbb C^{2\times 2}$ be given by
\[
P=\begin{bmatrix}1/\sqrt{2}&-1/\sqrt{2}\\ -1/\sqrt{2}& -1/\sqrt{2}\end{bmatrix}.
\]
Then, $P\in\mathbb{O}(2,\mathbb C)$. It is a direct calculation to show that
\[
\mathcal B:=P\cdot\mathcal A\in\mathbb T(\mathbb C^2,3)
\]
with
\[
b_{111}=\frac{-1}{2\sqrt{2}}\ \text{and }b_{222}=\frac{-1}{2\sqrt{2}}.
\]
Henceforth, 
\[
\operatorname{tr}(\mathcal B):=2(b_{111}+b_{222})=-\sqrt{2}. 
\]
It follows from \cite[Section~6]{hhlq13} that
\[
\operatorname{Det}(\lambda\mathcal I-\mathcal B)=\lambda^4-\operatorname{tr}(\mathcal B)\lambda^3+\text{lower order terms}=
\lambda^4+\sqrt{2}\lambda^3+\text{lower order terms}.
\]
So, $\mathcal B=P\cdot\mathcal A$ has a nonzero eigenvalue, and hence it has the algebraic multiplicity of the zero eigenvalue being strictly less than $4$. 
\end{proof}

A direct consequence of Proposition~\ref{prop:noninvariant} is that the multiplicity of any eigenvalue of a tensor could change when we apply the above group action. So, the algebraic multiplicity is not a good geometric object, and it becomes quite mixed on the orbits in $\mathbb T(\mathbb C^n,m)$ by the natural action of $\mathbb{O}(n,\mathbb C)$. 
Though this fact, we can still ask for the relationship between the algebraic multiplicity and the geometric multiplicity. In general, the philosophy, together with the classical linear algebra, suggests that
\begin{equation}\label{relation}
\text{algebraic multiplicity } \geq \ \mathfrak f\big(\text{geometric multiplicity}\big)\geq \ \text{geometric multiplicity},
\end{equation}
for some function $\mathfrak f$ which depends also on the order $m$ of the tensors in the ambient space. 

On the other hand, for the zero eigenvalue, the geometric multiplicity is ``geometric". 
%---------------------------
\begin{proposition}[Zero Geometric Multiplicity]\label{prop:geoinvariant}
The number of irreducible components of the eigenvariety of the zero eigenvalue and their dimensions are invariants under the group action by $\mathbb{O}(n,\mathbb C)$ as above. 
Especially, the geometric multiplicity of the zero eigenvalue of a tensor is an invariant.
\end{proposition}

\begin{proof}
The eigenvariety $V_{P\cdot\mathcal A}(0)$ of the tensor $P\cdot\mathcal A$ for the eigenvalue zero 
is the transformation $PV_{\mathcal A}(0)$ of the eigenvariety $V_{\mathcal A}(0)$ of $\mathcal A$ for the eigenvalue zero.
Since both the dimension and the number of irreducible components of a variety are invariants under coordinate changes, the result follows. 
\end{proof}

In linear algebra, since the eigenvariety is an eigenspace, i.e., a linear subspace, under the orthogonal group action, the unique essential geometric characterization is the dimension of that eigenspace. However, the eigenvariety for the tensor case is much more complicated, as nonlinear algebra to linear algebra. Nevertheless, the first two essential characterizations of the eigenvariety would be the number of irreducible components and their dimensions.

%%%%%%%%%%%%%%%%%%%
\subsection{A Conjecture on the Relationship}\label{sec:relation}
As suggested by Proposition~\ref{prop:geoinvariant}, we continue with the example in Proposition~\ref{prop:noninvariant}, and compute out the other entries of $\mathcal B$ as
\[
b_{112}=-\frac{1}{2\sqrt{2}},\ b_{121}=\frac{1}{2\sqrt{2}},\ b_{122}=\frac{1}{2\sqrt{2}},\ b_{211}=\frac{1}{2\sqrt{2}},
\ b_{212}=\frac{1}{2\sqrt{2}}\ \text{and } b_{221}=-\frac{1}{2\sqrt{2}}.
\]
So, the equations for the eigenvalue problem of the tensor $\mathcal B$ are
\[ 
\frac{1}{2\sqrt{2}}\big(-x^2+y^2\big)=\lambda x^2,\ \text{and }\frac{1}{2\sqrt{2}}\big(x^2-y^2\big)=\lambda y^2.
\]
By Macaulay 2 \cite{gs}, we can compute the characteristic polynomial of $\mathcal B$ which is
\[
\chi(\lambda)=\lambda^4+\sqrt{2}\lambda^3+\frac{1}{2}\lambda^2.
\]
Therefore, the eigenvalues of $\mathcal B$ are
\[
0 (\text{with }\operatorname{am}(0)=2)\ \text{and }-\frac{1}{\sqrt{2}}(\text{with }\operatorname{am}(-\frac{1}{\sqrt{2}})=2).
\]

Since $\mathcal A\neq 0$, it can be shown that the geometric multiplicity of $\mathcal A$ for the zero eigenvalue cannot be two, and hence it is one. 
By direct calculation, we have that the eigenvariety of $\mathcal A$ is
\[
V(0)=\mathbb C \{(1,0)\}\cup \mathbb C \{(0,1)\},
\]
which has two irreducible components. We also see that the algebraic multiplicity of the zero eigenvalue of $\mathcal B$ is the number of irreducible components of the eigenvariety of the zero eigenvalue with each irreducible component being a point in the projective space and hence dimension one in the affine space. 

In general, for every irreducible component $V_i$ of the eigenvariety of an eigenvalue of a tensor, there is an intrinsic number $\operatorname{ess}(V_i)$ associated to it such that
\[
\operatorname{am}(\lambda)\geq \operatorname{ess}(V_i)\geq\operatorname{dim}(V_i),
\]
where the second inequality is suggested by Proposition~\ref{prop:geoinvariant} since $\operatorname{dim}(V_i)$ is an invariant, and some
examples as well as the linear algebra. Also by the same proposition, we should have
\begin{equation}\label{general-relation}
\operatorname{am}(\lambda)\geq \sum_{V_i\subset V(\lambda)\ \text{is an irreducible component}}\operatorname{ess}(V_i).
\end{equation}
Obviously, $\operatorname{ess}(V_i)$ depends on the order $m$ of the ambient tensor space in which the eigenvalue problem is studied. Whenever we have \eqref{general-relation}, we have a realization of \eqref{relation} since
\[
\operatorname{am}(\lambda)\geq \sum_{V_i\subset V(\lambda)\ \text{is an irreducible component}}\operatorname{ess}(V_i)\geq \operatorname{ess}(V_*)\geq \operatorname{gm}(\lambda)
\]
with $\operatorname{dim}(V_*)=\operatorname{gm}(\lambda)$.
In linear algebra (i.e., $m=2$), $V(\lambda)$ is connected, and
\[
\operatorname{ess}(V(\lambda))=\operatorname{dim}(V(\lambda)). 
\]
We conjecture that
\[
\operatorname{ess}(V_i)=\operatorname{dim}(V_i)(m-1)^{\operatorname{dim}(V_i)-1},
\]
which is satisfied by the former examples, and will be supported by all the cases studied subsequently. 
%-----------------------
\begin{conjecture}[Conjecture]\label{con:relation}
Let $\mathcal T\in\mathbb T(\mathbb C^n,m)$ and $\lambda\in\sigma(\mathcal T)$. Suppose that $V_i: i=1,\dots,\kappa$ are the irreducible components of the eigenvariety $V(\lambda)$ and their dimensions are $\operatorname{dim}(V_i), i=1,\dots,\kappa$ respectively. Then, we have that
\[
\operatorname{am}(\lambda)\geq \sum_{i=1}^\kappa \operatorname{dim}(V_i)(m-1)^{\operatorname{dim}(V_i)-1}\geq \operatorname{gm}(\lambda)(m-1)^{\operatorname{gm}(\lambda)-1}.
\]
\end{conjecture}

The trivial case serves as the first taste of the conjecture.
%-----------------
\begin{proposition}[Identity Tensor]\label{prop:identity}
Let tensor $\mathcal I\in\mathbb{T}(\mathbb{C}^n,m)$ be the identity tensor. Then, the characteristic polynomial of $\mu\mathcal I$ is
\[
\chi(\lambda)=(\lambda-\mu)^{n(m-1)^{n-1}}
\]
for any $\mu\in\mathbb C$.
\end{proposition}

For the tensor $\mu\mathcal I$,
from Definition~\ref{def:eigenvalue-eigenvector}, we see that $V(\mu)=\mathbb C^n$, $\sigma(\mu\mathcal I)=\{\mu\}$ and hence $\operatorname{gm}(\mu)=n$; and from Proposition~\ref{prop:identity} that $\operatorname{am}(\mu)=n(m-1)^{n-1}=\operatorname{gm}(\mu)(m-1)^{\operatorname{gm}(\mu)-1}$. 
%%%%%%%%%%%%%%%%%%%%%%%%%%%
\section{Lower Rank Tensors}\label{sec:lower}
We continue to study the zero eigenvalue for Conjecture~\ref{con:relation} in this section. The underlying tensors are lower rank symmetric tensors. 

A tensor $\mathcal T\in\mathbb T(\mathbb C^n,m)$ is \textit{symmetric}, if $t_{i_1\dots i_m}=t_{i_{\tau(1)}\dots i_{\tau(m)}}$ for all $\tau\in\mathfrak S(m)$, the permutation group on $m$ elements. 
Let $\operatorname{\mathcal S}^{m}(\mathbb C^n)\subset \mathbb T(\mathbb C^n,m)$ be the space of $m$-th order $n$-dimensional symmetric tensors and $\mathcal A\in\operatorname{\mathcal S}^{m}(\mathbb C^n)$. Every tensor $\mathcal A\in \operatorname{\mathcal S}^{m}(\mathbb C^n)$ has a symmetric rank one decomposition as \cite{l12}
\begin{equation}\label{rankone}
\mathcal A=\sum_{i=1}^R\mathbf a_i^{\otimes m},
\end{equation}
where $\mathbf a_i\in\mathbb C^n$ and $R\in\mathbb N$ is a nonnegative integer. 
With the rank one decomposition \eqref{rankone}, the eigenvalue equations of the tensor $\mathcal A$ become
\begin{equation}\label{eigenvalue}
\sum_{i=1}^R(\mathbf a_i^\mathsf{T}\mathbf x)^{m-1}\mathbf a_i=\lambda \mathbf x^{[m-1]}.
\end{equation}
Let
\[
A=[\mathbf a_1,\dots, \mathbf a_R],\ \text{and } \mathbf d=\big((\mathbf a_1^\mathsf{T}\mathbf x)^{m-1},\dots, (\mathbf a_R^\mathsf{T}\mathbf x)^{m-1}\big)^\mathsf{T}.
\]
The rank of the matrix $A\in\mathbb C^{n\times R}$ is called the \textit{marginal rank} of the tensor $\mathcal A$, which is denoted by $\operatorname{mrank}(\mathcal A)$. The column space of the matrix $A$ is the \textit{marginal space} of the tensor $\mathcal A$. 
It is easy to see that the marginal space and hence the marginal rank of a tensor are indepedent of the rank decomposition \eqref{rankone}, see also \cite{l12} and references herein.

A tensor $\mathcal A$ is called \textit{essentially orthogonal decomposable}, if its marginal rank is $R$, i.e., the matrix $A$ is of full rank. 
Since $A$ consists of linear independent vectors, the tensor $\mathcal A$ can be decomposed as $P\cdot\mathcal B$ with $P\in\mathbb{GL}(n,\mathbb C)$ and $\mathcal B$ an \textit{orthogonal decomposable tensor}, i.e., a tensor possessing a rank one decomposition \eqref{rankone} with an orthogonal matrix $A$. Orthogonal decomposable tensors have important applications in machine learning, see \cite{r14} and references herein.

Let $\operatorname{mrank}(\mathcal A)=s<n$. 
An immediate consequence of \eqref{eigenvalue} is that the eigenvariety $V(0)$ of $\mathcal A$ contains the kernel of $A^\mathsf{T}$. Hence, $\operatorname{gm}(0)\geq n-s$. 
%------------------------------------------------------------
\begin{proposition}[Generic Geometric Multiplicity]\label{prop:zero-rank}
Let $\mathcal A\in \operatorname{\mathcal S}^{m}(\mathbb C^n)$ with $\operatorname{mrank}(\mathcal A)=s\leq n$ be generic. Then
\[
V(0)=\operatorname{ker}(A^\mathsf{T}),
\]
and hence $\operatorname{gm}(0)=n-s$. 
\end{proposition}

\begin{proof}
The case when $s=n$ is trivial. We show the other cases.
Let $P\in\mathbb{O}(n,\mathbb C)$ be a nonsingular matrix such that $P\cdot \mathcal A$ is in an upper block diagonal form $\mathcal B$ with the nonzero block being a tensor  $\mathcal C$ in $\operatorname{\mathcal S}^{m}(\mathbb C^s)$, i.e.,
\[
b_{i_1\dots i_m}=\begin{cases}c_{i_1\dots i_m} & \text{if }i_1,\dots,i_m\in\{1,\dots,s\},\\0 &\text{otherwise}.\end{cases}
\]
Therefore,
\[
PA=\begin{bmatrix}B\\\mathbf 0\end{bmatrix}
\]
with some $B\in\mathbb C^{s\times R}$. 
Let
\[
P=[\mathbf p_1,\dots,\mathbf p_n]^\mathsf{T}.
\]
We adopt the coordinate change $\mathbf x\mapsto P^\mathsf{T}\mathbf y$. 
Under this transformation, by \eqref{group-equ},
the equations
\[
P\big[\mathcal A\mathbf x^{m-1}\big]=\mathbf 0
\]
become
\[
P\big[\mathcal A(P^\mathsf{T}\mathbf y)^{m-1}\big]=(P\cdot\mathcal A)(\mathbf y)^{m-1}=
\mathcal B\mathbf y^{m-1} =\mathbf 0,
\]
or explicitly as
\begin{equation}\label{zero-solution}
\mathcal C\mathbf z^{m-1}=\mathbf 0
\end{equation}
with $\mathbf y=(\mathbf z,\mathbf w)$ and $\mathbf z\in\mathbb C^s$.
We see that
\[
\mathbf x\in V_{\mathcal A}(0) \ \text{if and only if } P\mathbf x\in V_{\mathcal B}(0)\ \text{if and only if } P\mathbf x=(\mathbf z,\mathbf w)\ \text{with }\mathbf z\in V_{\mathcal C}(0).
\] 

Note that the set of symmetric tensors with marginal rank being $t\leq s<n$ is a variety in the total space $\operatorname{\mathcal S}^{m}(\mathbb C^n)$ and it is invariant under the group action by $\mathbb{O}(n,\mathbb C)$.
For a generic tensor $\mathcal A$ in this variety, it has the maximal marginal rank $s$ and 
\[
\operatorname{Det}(\mathcal C)\neq 0,
\]
since the tensor $\binom{\mathcal I_s}{\mathbf 0}$ obviously has the marginal rank $s$ and $\operatorname{Det}(\mathcal I_s)=1
$. Here $\mathcal{I}_s$ is the identity tensor in $\mathcal{S}^s (\mathbb{C}^s)$. 

Therefore, for a generic $\mathcal A\in\operatorname{\mathcal S}^{m}(\mathbb C^n)$ of marginal rank $s$, the unique solution to \eqref{zero-solution} is $\mathbf z=\mathbf 0$. So,
\[
\mathbf x\in V_{\mathcal A}(0) \ \text{if and only if }  P\mathbf x=(\mathbf 0,\mathbf w)\ \text{with }\mathbf w\in \mathbb C^{n-s}. 
\] 
By the shape of $PA$, it follows that
\[
V_{\mathcal A}(0) =\operatorname{ker}(A^\mathsf{T}),
\]
and the results follow. 
\end{proof}
 
The next theorem says that Conjecture~\ref{con:relation} is true for the zero eigenvalue of symmetric tensors of a fixed marginal rank. 
%----------------------------------
\begin{theorem}[Lower Marginal Rank Symmetric Tensors]\label{thm:nonzeroeigenvalue}
Let $\mathcal A\in \operatorname{\mathcal S}^{m}(\mathbb C^n)$ and $\operatorname{mrank}(\mathcal A)=s$. Then, the number $\operatorname{nnz}(\mathcal A)$ of nonzero eigenvalues of $\mathcal A$ satisfies
\begin{equation}\label{nonzeronumber}
\operatorname{nnz}(\mathcal A)\leq s (m-1)^{n-1}
\end{equation}
with equality holding for generic tensors of marginal rank $\operatorname{mrank}(\mathcal A)=s$. Therefore, we have that
\[
\operatorname{am}(0)\geq (n-s)(m-1)^{n-1}\geq \operatorname{gm}(0)(m-1)^{\operatorname{gm}(0)-1},
\]
and a generic tensor of marginal rank $s$ has zero algebraic multiplicity exactly $ (n-s)(m-1)^{n-1}$. 
\end{theorem}

\begin{proof}
Similarly as in the proof for Proposition~\ref{prop:zero-rank}, let $P\in\mathbb{O}(n,\mathbb C)$ be a nonsingular matrix such that $P\cdot \mathcal A$ is in an upper block diagonal form $\mathcal B$ with the nonzero block being a tensor  $\mathcal C$ in $\operatorname{\mathcal S}^{m}(\mathbb C^s)$, and
\[
P=[\mathbf p_1,\dots,\mathbf p_n]^\mathsf{T}.
\]
Under the coordinate change $\mathbf x\mapsto P^\mathsf{T}\mathbf y$, 
the equations
\[
P\big[\mathcal A\mathbf x^{m-1}-\lambda\mathbf x^{[m-1]}\big]=\mathbf 0
\]
become (cf. \eqref{group-equ})
\[
P\big[\mathcal A(P^\mathsf{T}\mathbf y)^{m-1}-\lambda(P^\mathsf{T}\mathbf y)^{[m-1]}\big]=\mathcal B\mathbf y^{m-1}-\lambda P\big[(P^\mathsf{T}\mathbf y)^{[m-1]}\big] =\mathbf 0.
\]
Let $\mathcal D\in \operatorname{\mathcal T}^{m}(\mathbb C^n)$ be the tensor associated to $P\big[(P^\mathsf{T}\mathbf y)^{[m-1]}\big]$, and be partitioned as 
\[
\mathcal D=\begin{bmatrix}\mathcal D_1\\ \mathcal D_2\end{bmatrix}
\]
with $\mathcal D_1$ corresponding to the first $s$ slices of $\mathcal D$ and $\mathcal D_2$ the others. It follows from \eqref{group-equ} that $\mathcal D=P\cdot\mathcal I$. 
Then,
\begin{align*}
\operatorname{Det}(P)^{m(m-1)^{n-1}}\operatorname{Det}(\lambda \mathcal I-\mathcal A)&=\operatorname{Det}\big(P\cdot(\lambda \mathcal I-\mathcal A)\big)\\
&=\operatorname{Det}(\lambda \mathcal D-\mathcal B)\\
&=\operatorname{Det}\bigg(\begin{bmatrix}\lambda \mathcal D_1-\mathcal C\\ \lambda \mathcal D_2\end{bmatrix}\bigg)\\
&=\lambda^{(n-s)(m-1)^{n-1}}\operatorname{Det}\bigg(\begin{bmatrix}\lambda \mathcal D_1-\mathcal C\\ \mathcal D_2\end{bmatrix}\bigg)
\end{align*}
where the equalities follow from \cite[Theorems~3.3.5 and 3.3.3.1]{clo98}. Therefore,
\[
\operatorname{Det}(\lambda \mathcal I-\mathcal A)
\]
has a factor of $\lambda^t$ with $t\geq (n-s) (m-1)^{n-1}$, since $\operatorname{Det}(P)\neq 0$. Henceforth, we have the bound for the number of nonzero eigenvalues \eqref{nonzeronumber}. 

Since $\operatorname{Det}(\lambda \mathcal I-\mathcal A)$ is a monic polynomial in $\lambda$ of degree $n(m-1)^{n-1}$ with constant term being $\operatorname{Det}(-\mathcal A)$, we have that
\[
\operatorname{Det}\bigg(\begin{bmatrix}\lambda \mathcal D_1-\mathcal C\\ \mathcal D_2\end{bmatrix}\bigg)=\operatorname{Det}(P)^{-m(m-1)^{n-1}}\lambda^{s(m-1)^{n-1}}+\text{lower order terms}.
\]
Moreover, $\operatorname{nnz}(\mathcal A)$ attains the upper bound if and only if
\begin{equation}\label{nonzero-res}
\operatorname{Det}\bigg(\begin{bmatrix}\mathcal C\\ \mathcal D_2\end{bmatrix}\bigg)\neq 0. 
\end{equation}
Note that for a generic $\mathcal A$, $\operatorname{Det}(\mathcal C)\neq 0$ (since $\mathcal I_s\in\operatorname{\mathcal S}^{m}(\mathbb C^s)$ has determinant $1$), therefore generically 
\[
\operatorname{Det}\bigg(\begin{bmatrix}\mathcal C\\ \mathcal D_2\end{bmatrix}\bigg)=0
\]
if and only if 
\[
\mathcal D_2\mathbf y^{m-1}=\mathbf 0
\]
has a nonzero solution in $\{\mathbf y\in\mathbb C^n :  y_1=\dots=y_s=0\}$. Let $P^\mathsf{T}$ be partitioned accordingly as
\[
P^\mathsf{T}=\begin{bmatrix}P_l^\mathsf{T}&P_r^\mathsf{T}\end{bmatrix}. 
\]
Then
\[
\mathcal D_2\mathbf y^{m-1}=\mathbf 0
\]
has a nonzero solution in $\{\mathbf y\in\mathbb C^n :  y_1=\dots=y_s=0\}$
if and only if
\begin{equation}\label{nonzero:eq1}
P_r (P_r^\mathsf{T}\mathbf w)^{[m-1]}=\mathbf 0
\end{equation}
has a nonzero solution in $\mathbb C^{n-s}$. Since $P_r^\mathsf{T}$ has full rank, we can assume without loss of generality that the first $n-s$ rows of $P_r^\mathsf{T}$ has full rank. This matrix is denoted by $P_0$. Then,
\[
P_r^\mathsf{T}=\begin{bmatrix}P_0\\ P_1\end{bmatrix}.
\]
Let $\mathbf z=P_0\mathbf w\in\mathbb C^{n-s}$. 
It follows that \eqref{nonzero:eq1} is equivalent to
\[
[P_0^\mathsf{T},P_1^\mathsf{T}]\big((\mathbf z^{[m-1]})^\mathsf{T}, ((P_1P_0^\mathsf{-1}\mathbf z)^{[m-1]})^\mathsf{T}\big)^\mathsf{T}=\mathbf 0
\]
and furthermore
\[
[I,P_0^\mathsf{-T}P_1^\mathsf{T}]\big((\mathbf z^{[m-1]})^\mathsf{T}, ((P_1P_0^\mathsf{-1}\mathbf z)^{[m-1]})^\mathsf{T}\big)^\mathsf{T}=\mathbf 0
\]
has a nonzero solution $\mathbf z$ in $\mathbb C^{n-s}$. 
Let $M=P_1P_0^\mathsf{-1}$. Then, we have that it can be further reformulated as
\[
\mathbf z^{[m-1]}+M^\mathsf{T}(M\mathbf z)^{[m-1]}=\mathbf 0.
\]
By \eqref{group-equ}, it is
\[
\big(\mathcal I+M^\mathsf{T}\cdot\mathcal I\big)\mathbf z^{m-1}=\mathbf 0. 
\]
For generic $P$, the tensor $M^\mathsf{T}\cdot\mathcal I$ does not have an eigenvalue $-1$, since it does when $P_0=I\in\mathbb C^{(n-s)\times(n-s)}$ and $P_1=\mathbf 0\in\mathbb C^{s\times(n-s)}$. Therefore, this system does not have a nontrivial solution in general and hence \eqref{nonzero-res} holds. 
\end{proof}

The next proposition is in deterministic form. 
%----------------------------
\begin{proposition}[Essentially Orthogonal Decomposable Tensors]\label{prop:lower}
Let $\mathcal A\in \operatorname{\mathcal S}^{m}(\mathbb C^n)$ be such that the matrix $A$ from \eqref{rankone} has full rank, i.e., $R=\operatorname{mrank}(\mathcal A)=s\leq n$. Then, 
\[
\operatorname{gm}(0)=n-s,
\] 
and therefore
\[
\operatorname{am}(0)\geq (n-s)(m-1)^{n-1}\geq \operatorname{ess}(V(0))=\operatorname{gm}(0)(m-1)^{\operatorname{gm}(0)-1}.
\]
\end{proposition}

\begin{proof}
Sine the matrix $A$ is of full rank, it follows from \eqref{eigenvalue} that $\mathbf x\in V(0)$ if and only if
\[
\mathbf d=\mathbf 0
\]
if and only if
\[
A^\mathsf{T}\mathbf x=\mathbf 0.
\]
Therefore, $V(0)$ is the kernel of the matrix $A^\mathsf{T}$, which is a linear subspace of dimension $n-s$. 
\end{proof}

%%%%%%%%%%%%%
\section{Coordinate Linear Subspace as Eigenvectors}\label{sec:linear}
In this section, we generalize the relation between the two multiplicities for matrices to tensors. As we saw from Section~\ref{sec:matrix}, the general relation for matrices follows from the case when the eigenspace is a linear subspace in coordinate form and the fact that both multiplicities are invariants under the orthogonal linear group action. Since the algebraic multiplicity for a tensor is not an invariant any more (cf.\ Proposition~\ref{prop:noninvariant}), we show the coordinate case. To this end, properties on quasi-triangular tensors and
symmetrizations of tensors are necessary.

%%%%%%%%%%%%%%
\subsection{Quasi-Triangular Tensors}\label{sec:quasi}
%-------------------------------------------------
\begin{definition}[Sub-Tensor]\label{def:sub-tensor}
Let $\mathcal T\in\mathbb{T}(\mathbb{C}^n,m)$ and $1\leq k\leq n$. Tensor $\mathcal U\in\mathbb{T}(\mathbb{C}^k,m)$ is called a \textit{sub-tensor} of $\mathcal T$
associated to the index set $\{j_1,\ldots,j_k\}\subseteq\{1,\ldots,n\}$ if and only if
$u_{i_1\ldots i_m}=t_{j_{i_1}\ldots j_{i_m}}$ for all $i_1,\ldots,i_m\in\{1,\ldots,k\}$.
\end{definition}
%-------------------------
\begin{proposition}[Quasi-Triangular Tensor]\label{prop:quasi-tri}
Let tensor $\mathcal T=(t_{ii_2\ldots i_m})\in\mathbb{T}(\mathbb{C}^n,m)$. Suppose that there exists some $k\in\{1,\dots,n\}$ such that $\sum_{(i_2,\dots,i_m)\in\mathbb X(\alpha)}t_{ii_2\dots i_m}=0$ for all $i>k$ and $\alpha\in\{0,\dots,m-1\}^k\times\{0\}^{n-k}$. Let $\mathcal U\in\mathbb T(\mathbb C^k,m)$ be the sub-tensor of $\mathcal T$ associated to $\{1,\dots,k\}$. Then,
\[
\operatorname{DET}(\mathcal T)=\operatorname{DET}(\mathcal U) p(\mathcal T)
\]
for some polynomial $p\in\mathbb C[\mathcal T]$. 
\end{proposition}

\begin{proof}
Let $\mathcal T=(t_{ii_2\ldots i_m})\in\mathbb{T}(\mathbb{C}^n,m)$ be a tensor such that the hypothesis being satisfied and $\operatorname{Det}(\mathcal U)=0$ for the sub-tensor $\mathcal U$. Then, by Proposition~\ref{prop:irreducible}, we see that there is a vector $\mathbf y\in\mathbb C^k\setminus\{\mathbf 0\}$ such that
\[
\mathcal U\mathbf y^{m-1}=\mathbf 0.
\]
Define a vector $\mathbf x\in\mathbb C^n$ as $\mathbf x=(\mathbf y^\mathsf{T},\mathbf 0^\mathsf{T})^\mathsf{T}$. Then $\mathbf x\neq \mathbf 0$, and by the assumption on $\mathcal T$ 
\[
\mathcal T\mathbf x^{m-1}=\big((\mathcal U\mathbf y^{m-1})^\mathsf{T},\mathbf 0^\mathsf{T}\big)^\mathsf{T}=\mathbf 0.
\]
Therefore, by Proposition~\ref{prop:irreducible} again, $\operatorname{Det}(\mathcal T)=0$. 

The set of  tensors satisfying the hypothesis is a linear subspace $\mathbb L$ of the total space $\mathbb{T}(\mathbb{C}^n,m)$. Let $\mathbb{M}$ be the set of sub-tensors of elements of $\mathbb{L}$ associated to $\{1,\dots,k\}$. There is a projection 
\[
\delta:\mathbb{T}(\mathbb{C}^n,m)\to \mathbb{T}(\mathbb{C}^k,m)
\] 
sending $\mathcal{T}\in \mathbb{T}(\mathbb{C}^n,m)$ to its sub-tensor $\mathcal{U}\in \mathbb{T}(\mathbb{C}^k,m)$ associated to $\{1,\dots,k\}$. In particular, $\delta$ maps $\mathbb{L}$ onto $\mathbb{M}$. Let $\delta^*:\mathbb{C}[\mathcal{U}]\to \mathbb{C}[\mathcal{T}]$ be the induced homomorphism. It is easy to see that $\mathbb{M}$ is in fact simply $\mathbb{T}(\mathbb{C}^k,m)$. 

Hence we have 
\[
\mathbb{L}\cap \mathbb{V} (\delta^*\operatorname{DET}(\mathcal{U})) \subseteq \mathbb{V}(\operatorname{DET}(\mathcal{T})).
\]
This implies that on the linear space $\mathbb{L}$ we have 
\[
\mathbb{V} (\delta^*\operatorname{DET}(\mathcal{U})) \subseteq \mathbb{V}(\operatorname{DET}(\mathcal{T})).
\]
It is not hard to see that $\delta^*\operatorname{DET}(\mathcal{U})$ is an irreducible polynomial even on $\mathbb{L}$, since $\mathbb{L}$ is the linear subspace consisting of quasi-triangular tensors. Therefore by Nullstellensatz \cite{clo98} we conclude that
\[
\operatorname{DET}(\mathcal T)=\operatorname{DET}(\mathcal U) p(\mathcal T)
\]
for some polynomial $p\in\mathbb C[\mathcal{T}]$.
\end{proof}

Tensors in $\mathbb T(\mathbb C^n,m)$ satisfying the hypothesis in Proposition~\ref{prop:quasi-tri} are \textit{quasi-triangular tensors}. It is a class of structured tensors broader than the class of upper triangular tensors introduced in \cite{hhlq13}. We note that both notions reduce to upper triangularity in the matrix counterpart.  

%%%%%%%%%%%%%%%%
\subsection{Symmetrization}\label{sec:symmetrization}
The topic of this article is eigenvalues of tensors, which are closely and solely related to $\mathcal T\mathbf x^{m-1}$. It is easy to see that the $i$-th slice  $\mathcal T_i:=(t_{ii_2\dots i_m})_{1\leq i_2,\dots,i_m\leq n}$ of $\mathcal T$ is an $(m-1)$-th order $n$-dimensional tensor, and for all $i=1,\dots,n$
\[
\big(\mathcal T\mathbf x^{m-1}\big)_i=\langle\operatorname{Sym}(\mathcal T_i),\mathbf x^{\otimes (m-1)}\rangle:=\sum_{i_2,\dots,i_m=1}^n\big(\operatorname{Sym}(\mathcal T_i)\big)_{i_2\dots i_m}x_{i_2}\dots x_{i_m}\ \text{for all }\mathbf x\in\mathbb C^n,
\]
where $\operatorname{Sym}(\mathcal T_i)$ is the \textit{symmetrization} of the tensor $\mathcal T_i$ as a symmetric tensor in the sense of the above equalities. Therefore, for every tensor $\mathcal T\in\mathbb T(\mathbb C^n, m)$, we associate it an element $\operatorname{eSym}(\mathcal T)$ in $\mathbb{TS}(\mathbb{C}^n,m):=\mathbb C^n\otimes\operatorname{\mathcal S}^{m-1}(\mathbb C^n)$ by symmetrizing its slices. Hence,
\[
\mathcal T\mathbf x^{m-1}=\operatorname{eSym}(\mathcal T)\mathbf x^{m-1}\ \text{for all }\mathbf x\in\mathbb C^n.
\]
We see that all tensors in the fibre of the above map have the same defining equations for the eigenvalue problem.
%---------------------------
\begin{proposition}[Symmetrization]\label{prop:symmetrization}
Let $\mathcal T\in\mathbb{T}(\mathbb{C}^n,m)$. Then, 
\[
\operatorname{Det}(\mathcal T-\lambda\mathcal I)=\operatorname{Det}(\operatorname{eSym}(\mathcal T)-\lambda\mathcal I).
\]
\end{proposition}

\begin{proof}
Suppose that $\mathcal T\in\mathbb{T}(\mathbb{C}^n,m)$ with symbolic entries, then we have that 
\[
\operatorname{Det}(\mathcal T)=\operatorname{DET}.
\]
By Definition~\ref{def:determinant}, we have
\[
\operatorname{Det}(\mathcal T)=\operatorname{Det}(\operatorname{eSym}(\mathcal T)).
\]
Henceforth, the result follows from the fact that $\operatorname{eSym}(\mathcal I)=\mathcal I$. 
\end{proof}

%---------------------------
\begin{proposition}[Correspondence]\label{prop:correspondence}
The tensor space $\mathbb{TS}(\mathbb{C}^n,m)$ and the set of systems of $n$ homogeneous polynomials in $n$ variables of degree $m-1$ has a one to one correspondence.
\end{proposition}

\begin{proof}
The correspondence is indicated by $\mathcal T\mathbf x^{m-1}$ for $\mathcal T\in\mathbb{TS}(\mathbb{C}^n,m)$.
\end{proof}

%%%%%%%%%%%%%%%
\subsection{The Coordinate Case}\label{sec:coord}
The eigenvariety can determine the tensor in certain case. 
%-------------------
\begin{proposition}[Unique]\label{prop:unique}
Let tensor $\mathcal T=(t_{ii_2\ldots i_m})\in\mathbb{TS}(\mathbb{C}^n,m)$. If 
$V(\lambda)=\mathbb C^n$ for some $\lambda$, then 
\[
\mathcal T=\lambda\mathcal I.
\]
\end{proposition}

\begin{proof}
Note that $V(\lambda)=\mathbb C^n$ is equivalent to say that the polynomial system
\[
(\mathcal T-\lambda \mathcal I)\mathbf x^{m-1}\equiv\mathbf 0.
\]
Therefore, every polynomial $\big((\mathcal T-\lambda \mathcal I)\mathbf x^{m-1}\big)_i\equiv 0$. By Proposition~\ref{prop:correspondence} and Hilbert's Nullstellensatz \cite{clo98}, we have that all the coefficients of this polynomial are zero.
Henceforth, we must have $\mathcal T=\lambda \mathcal I$.
\end{proof}

The next theorem is the main theorem in this section. 
%-------------------
\begin{theorem}[Coordinate Case]\label{thm:multiplicity}
Let tensor $\mathcal T\in\mathbb T(\mathbb{C}^n,m)$. If for some $\lambda\in\sigma(\mathcal T)$, $V(\lambda)\supseteq P\{\mathbf x\in\mathbb C^n : x_{\operatorname{gm}(\lambda)+1}=\dots=x_{n}=0\}$ for some permutation matrix $P\in\mathbb{O}(\mathbb C,n)$, then
we have that
\[
\operatorname{am}(\lambda)\geq \operatorname{gm}(\lambda)(m-1)^{\operatorname{gm}(\lambda)-1}.
\]
\end{theorem}

\begin{proof}
Let $P\in\mathbb{O}(n,\mathbb C)$ be a permutation matrix. It is easy to see that
\[
P\cdot \mathcal I=\mathcal I.
\]
Therefore,
\[
\operatorname{Det}(\lambda \mathcal I-P\cdot \mathcal A)=
\operatorname{Det}(P\cdot(\lambda \mathcal I-\mathcal A))=
\operatorname{Det}(\lambda \mathcal I-\mathcal A). 
\]
Henceforth, we can assume that $P=I$, the identity matrix, without loss of any generality. 

Let $\mathcal B=\operatorname{eSym}(\mathcal T)$ be the tensor in $\mathbb{TS}(\mathbb C^n,m)$ corresponding to $\mathcal T$. Since $\mathcal B$ and $\mathcal T$ have the same defining equations for the eigenvalue problem, they have the same eigenvalues and eigenvectors. Let $\mathcal U\in\mathbb T(\mathbb C^{ \operatorname{gm}(\lambda)},m)$ and $\mathcal C\in\mathbb {TS}(\mathbb C^{ \operatorname{gm}(\lambda)},m)$ be the sub-tensors of $\mathcal T$ and $\mathcal B$ associated to $\{1,\dots, \operatorname{gm}(\lambda)\}$ respectively. It is easy to see that $\mathcal C=\operatorname{eSym}(\mathcal U)$. 

By the hypothesis, 
\[
\mathcal B\mathbf x^{m-1}=\lambda\mathbf x^{[m-1]}\ \text{for all }\mathbf x\in \{\mathbf x\in\mathbb C^n : x_{\operatorname{gm}(\lambda)+1}=\dots=x_{n}=0\}.
\]
By a direct caculation, we see that
\[
\mathcal C\mathbf y^{m-1}=\lambda\mathbf y^{[m-1]}\ \text{for all }\mathbf y\in\mathbb C^{\operatorname{gm}(\lambda)}. 
\]
Hence, $\mathcal C$ has an eigenvalue $\lambda$ with the eigenvectors being the whole $\mathbb C^{\operatorname{gm}(\lambda)}\setminus\{\mathbf 0\}$. Therefore, it follows from Proposition~\ref{prop:unique} that 
\[
\mathcal C=\lambda\mathcal I. 
\]
 It follows from Propositions~\ref{prop:identity} and \ref{prop:symmetrization} that
\[
\operatorname{Det}(\mathcal U-\mu\mathcal I)=\operatorname{Det}(\mathcal C-\mu\mathcal I)=(\lambda-\mu)^{\operatorname{gm}(\lambda)(m-1)^{\operatorname{gm}(\lambda)-1}}.
\]

Likewise, by the hypothesis, we have that for all $i>\operatorname{gm}(\lambda)$,
\[
(\mathcal B\mathbf x^{m-1})_i=\sum_{i_2,\dots,i_m}b_{ii_2\dots i_m}x_{i_2}\dots x_{i_m}=\lambda x_i^{m-1}=0\ \text{for all }\mathbf x\in\{\mathbf x\in\mathbb C^n : x_{\operatorname{gm}(\lambda)+1}=\dots=x_n=0\}.
\]
Then, it follows from Proposition~\ref{prop:correspondence} that
\[
b_{ii_2\dots i_m}=0\ \text{for all }i>\operatorname{gm}(\lambda), \ i_2,\dots,i_m\in\{1,\dots,\operatorname{gm}(\lambda)\}.
\]
This is the same the say that $\sum_{(i_2,\dots,i_m)\in\mathbb X(\alpha)}t_{ii_2\dots i_m}=0$ for all $i>\operatorname{gm}(\lambda)$ and $\alpha\in\{0,\dots,m-1\}^{\operatorname{gm}(\lambda)}\times\{0\}^{n-\operatorname{gm}(\lambda)}$. In other words, both $\mathcal B$ and $\mathcal T$ are quasi-triangular and the associated index set is $\{1,\dots,\operatorname{gm}(\lambda)\}$.

Therefore,
it is easy to see that $\mathcal T-\mu\mathcal I$ satisfies the hypothesis in Proposition~\ref{prop:quasi-tri} with $k=\operatorname{gm}(\lambda)$. 
Henceforth, 
\[
\operatorname{Det}(\mathcal T-\mu\mathcal I)=\operatorname{Det}(\mathcal U-\mu\mathcal I)p(\mathcal T-\mu\mathcal I)=(\lambda-\mu)^{\operatorname{gm}(\lambda)(m-1)^{\operatorname{gm}(\lambda)-1}}p(\mathcal T-\mu\mathcal I)
\]
for some polynomial $p$.
So, $\lambda$ is an eigenvalue of $\mathcal T$ with algebraic multiplicity at least 
$\operatorname{gm}(\lambda)(m-1)^{\operatorname{gm}(\lambda)-1}$. The conclusion then follows. 
\end{proof}

%%%%%%%%%%%%%%%%%%%%%
\section{Generic Tensors}\label{sec:generic}
In this section, we consider generic tensors. 
%---------------------------%-------------------
\begin{lemma}[Generic Tensor Eigenvalue]\label{lem:generic}
Let tensor $\mathcal T\in\mathbb{T}(\mathbb{C}^n,m)$ be generic. Then, 
\[
\operatorname{am}(\lambda)=1
\]
for all $\lambda\in\sigma(\mathcal T)$. 
\end{lemma}

\begin{proof}
The polynomial $\operatorname{Det}(\lambda \mathcal I-\mathcal T)$ is monic and irreducible when we regard $\mathcal T$ symbolically \cite{hhlq13}. Therefore, for a generic tensor $\mathcal T$, each root of the characteristic polynomial is algebraically simple. 
It then follows from Proposition~\ref{prop:irreducible} that for generic tensor $\mathcal T$, every eigenvalue has algebraic multiplicity one. 
\end{proof}

%---------------------------%-------------------
\begin{lemma}[Generic Tensor Eigenvector]\label{lem:generic-eigenvector}
Let tensor $\mathcal T\in\mathbb{T}(\mathbb{C}^n,m)$ be generic. Then, 
\[
\operatorname{gm}(\lambda)=1
\]
for all $\lambda\in\sigma(\mathcal T)$. 
\end{lemma}

\begin{proof}
Given a tensor $\mathcal T\in\mathbb{T}(\mathbb{C}^n,m)$ we consider the eigen-equation system
\[
\mathcal{T}x^{m-1}=\lambda x^{[m-1]}.
\]
This is a polynomial system of $n$ inhomogeneous equations of degree $m$ in $(n+1)$ variables $x_1,\dots, x_n, \lambda$. We homogenize the system so variables are $X_1,\dots,X_n,\Lambda, Z$ such that $x_i=X_i/Z$ and $\lambda=\Lambda/Z$. Generically, the new homogeneous system defines a curve $\mathcal{C}$ in $\mathbb{P}^{n+1}:=\mathbb P\mathbb C^{n+1}$. Notice that whenever $[X_1,\dots,X_n,\Lambda, Z]$ is a solution, $[tX_1,\dots, tX_n, s\Lambda, sZ]$ is also a solution for any $t,s\in \mathbb{C}$. This implies that all irreducible components are parametrized as   $[tX_1,\dots, tX_n, s\Lambda, sZ]$ where $[X_1,\dots,X_n,\Lambda, Z]$ is a point on this component. Correspondingly, this implies that irreducible components of the original system are parametrized by $(t/sx_1,\dots, t/sx_n,\lambda)$. In particular, this shows that for a generic tensor $\mathcal{T}$, the dimension of the eigen-variety of any $\lambda\in \sigma(\mathcal{T})$ is one.
 
\end{proof}

%---------------------------%-------------------
\begin{lemma}[Unique Eigenvector]\label{lem:unique-eigenvector}
Let tensor $\mathcal T\in\mathbb{T}(\mathbb{C}^n,m)$ be generic. Then
$V(\lambda)$ has dimension one and is irreducibe for all $\lambda\in\sigma(\mathcal T)$, i.e., $\mathcal T$ has a unique (up to scaling) eigenvector for every $\lambda\in\sigma(\mathcal T)$.  
\end{lemma}

\begin{proof}

Let 
\[
I=\big\langle f_i(\mathbf x,\lambda):=(\mathcal T\mathbf x^{m-1})_i-\lambda x_i^{m-1}, \ i=1,\dots,n\big\rangle
\]
be the ideal generated by the eigenvalue-eigenvector polynomials. 
Define $t_i:=\frac{x_i}{x_n}$ for $i=1,\dots,n-1$. Let
\[
g_i(\mathbf t,\lambda)=f_i((\mathbf t,1),\lambda),\ \text{for all }i=1,\dots,n,
\]
and
\[
\mathbb V_0=\mathbb V(g_1,\dots,g_n).
\]
Then, $\mathbb V_0$ is the intersection of $\mathbb V(I)$ and $\{(\mathbf x,\lambda) : x_n=1\}$. 

Since $\mathcal T$ is generic, $\mathbb V(I)$ has dimension one and $\sigma(\mathcal T)$ is of cardinality $n(m-1)^{n-1}$, i.e., all the eigenvalues are distinct (cf.\ Lemma~\ref{lem:generic}). 
Therefore,
$\mathbb V_0$ is zero-dimensional and has distinct $\lambda$-coordinate.
For a generic tensor $\mathcal T$, it follows from a standard algebraic geometry argument that $\langle g_1,\dots,g_n\rangle$ is a radical ideal. Henceforth, we have that the reduced Gr\"obner basis of $\langle g_1,\dots,g_n\rangle$  in the lexicographic term order has the following shape (cf. the Shape Lemma (e.g., \cite[Proposition~2.3]{s02}))
\[
\mathcal G=\{t_1-q_1(\lambda),\dots,t_{n-1}-q_{n-1}(\lambda),q_n(\lambda)\}.
\]
When $\mathcal T$ is generic, $\mathbb V(I)$ is a union of finite irreducible compoents of dimension one and many singletons which will be cut off by $\mathbb V_0$, and then we have that $\#(\mathbb V_0)=n(m-1)^{n-1}=\operatorname{deg}(q_n)$ and $q_n(\lambda)=\chi(\lambda)$, the characteristic polynomial of $\mathcal T$. Therefore, 
$V(\lambda)$ has dimension one and is irreducibe for all $\lambda\in\sigma(\mathcal T)$, i.e., it has a unique (up to scaling) eigenvector.  
\end{proof}

In conclusion, we have the next theorem. 
%-------------
\begin{theorem}[Generic Tensor]\label{thm:generic}
Let tensor $\mathcal T\in\mathbb{T}(\mathbb{C}^n,m)$ be generic. Then, 
\[
\operatorname{am}(\lambda)=\operatorname{gm}(\lambda)=1
\]
for all $\lambda\in\sigma(\mathcal T)$.
\end{theorem}
Note that, in the generic case, we then always have the relation
\[
\operatorname{am}(\lambda)=\operatorname{gm}(\lambda)(m-1)^{\operatorname{gm}(\lambda)-1}.
\]

%%%%%%%%%%%%%%%%%%
\section{Duality}\label{sec:duality}
The set of symmetric tensors of order $m$ and dimension $n$ forms a strictly smaller subspace of the space of tensors of order $m$ and dimension $n$. Therefore, results in Section~\ref{sec:generic} do not apply directly to generic symmetric tensors which will be considered in this section. 

We consider the natural duality theory associated to eigenvalue theory in this section. 
We work in the space $\operatorname{\mathcal S}^m(\mathbb C^n)$ of $m$-th order and $n$-dimensional symmetric tensors. 

For the eigenvalues of symmetric tensors, we have a geometric interpretation in terms of the duality between the determinantal hypersurface and the Veronese variety. We refer to \cite{l12,h77,gkz94} for basic definitions. 

The $m$-th Veronese map $\nu_m$ from $\mathbb C^n$ to $\operatorname{\mathcal S}^m(\mathbb C^n)$ is 
\[
\nu_m(\mathbf x)=\mathbf x^{\otimes m}\ \text{for all }\mathbf x\in\mathbb C^n.
\]
The image of $\nu_m$ over $\mathbb C^n$ is the \textit{Veronese variety} $\nu_m(\mathbb C^n)$. This map is defined on $\mathbb{PC}^{n-1}$ naturally, and the image lies in the projective space $\mathbb P\operatorname{\mathcal S}^m(\mathbb C^n)$. The corresponding variety $\nu_m(\mathbb{PC}^{n-1})$ is a smooth, non-degenerate, homogeneous and irreducible variety of dimension $n-1$ \cite{l12}. We know that the tangent space of $\nu_m(\mathbb{PC}^{n-1})$ at a point $[\mathbf x^{\otimes m}]$ \footnote{For a variety $X\in\mathbb P V$ over the projective space of a linear space $V$, we write $[\mathbf y]$ for the equivalent class of the point $\mathbf y$ in the affine cone of $X$.} is
\[
T_{[\mathbf x^{\otimes m}]}\nu_m(\mathbb{PC}^{n-1})=\mathbb P\{\operatorname{Sym}(\mathbf x^{\otimes m-1}\otimes\mathbf y) : \mathbf y\in\mathbb C^n\}.
\]
If we identify the dual of $\mathbb P\operatorname{\mathcal S}^m(\mathbb C^n)$ as itself, and a point $H\in \mathbb P\operatorname{\mathcal S}^m(\mathbb C^n)$ as the hyperplane it defines in the tutorial way, the dual variety of $\nu_m(\mathbb{PC}^{n-1})$ is
\[
\nu_m(\mathbb{PC}^{n-1})^\mathsf{\vee}:=\{H\in \mathbb P\operatorname{\mathcal S}^m(\mathbb C^n) : T_{[\mathbf x^{\otimes m}]}\nu_m(\mathbb{PC}^{n-1})\subset H\ \text{for some }[\mathbf x^{\otimes m}]\}. 
\]
Let $\operatorname{SDET}$ be the restriction of the determinant $\operatorname{DET}$ on the symmetric tensor space. It is called \textit{symmetric hyperdeterminant} \cite{q05}. 
By the definition of determinant, it is easy to see that 
\[
\nu_m(\mathbb{PC}^{n-1})^\mathsf{\vee}=\mathbb V(\operatorname{SDET})=\mathbb V(\operatorname{DET})\cap\mathbb P\operatorname{\mathcal S}^m(\mathbb C^n) \subset\mathbb P\operatorname{\mathcal S}^m(\mathbb C^n),
\]
where the second equality follows from the factorization of the hyperdeterminant \cite{o12}. 
Therefore, geometrically, a nonzero tensor $\mathcal A\in \operatorname{\mathcal S}^m(\mathbb C^n)$ has an eigenvector $\mathbf x$ corresponding to eigenvalue zero if and only if the determinantal hypersurface (discriminant hypersurface in \cite{l12}) $\mathbb V(\operatorname{SDET})$ is tangent to the variety $\nu_m(\mathbb{PC}^{n-1})$ at the pair $([\mathcal A],[\mathbf x^{\otimes m}])$. 

Note that $(\lambda,\mathbf x)$ is an eigenpair of $\mathcal A$ if and only if $\mathbf x$ is an eigenvector of $\mathcal A-\lambda\mathcal I$ corresponding to eigenvalue zero. 

For every tensor $\mathcal A\in \operatorname{\mathcal S}^m(\mathbb C^n)$, 
\[
l(\lambda)=[\mathcal A-\lambda\mathcal I]
\]
defines a line in $\mathbb P\operatorname{\mathcal S}^m(\mathbb C^n)$. 
The next proposition is immediate.
%-----------------------------------------
\begin{proposition}[Intersection Interpretation]\label{prop:eigenvalue-dual}
Let $\mathcal A\in \operatorname{\mathcal S}^m(\mathbb C^n)$. Then, $\lambda\in\sigma(\mathcal A)$ is an eigenvalue of $\mathcal A$ if and only if the line $l(\gamma)=[\mathcal A-\gamma\mathcal I]$ has an intersection with the determinantal hypersurface $\mathbb V(\operatorname{SDET})$ at $\gamma=\lambda$. 
\end{proposition}

The hypersurface $\mathbb V(\operatorname{SDET})$ has degree $n(m-1)^{n-1}$ \cite{gkz94,q05}. With multiplicity, 
the intersections of the line $l(\lambda)$ with the hypersurface $\mathbb V(\operatorname{SDET})$ are exactly the eigenvalues of the tensor $\mathcal A$.  

%-----------------------------------------
\begin{proposition}[Generic Symmetric Tensor]\label{prop:generic-dual}
Let $\mathcal A\in \operatorname{\mathcal S}^m(\mathbb C^n)$ be generic. Then, every eigenvalue has both algebraic multiplicity and geometric multiplicity one. 
\end{proposition}

\begin{proof}
The symmetric hyperdeterminant $\operatorname{SDET}$ is still irreducible \cite{q05,cs13}. The characteristic polynomial $\operatorname{Det}(\mathcal T-\lambda\mathcal I)$ for a symmetric tensor is still monic. Therefore, we can show that for a generic tensor all the eigenvalues have algebraic multiplicity one. 

The dual of $\nu_m(\mathbb{PC}^{n-1})$ is an irreducible hypersurface \cite[Proposition 1.1.1.3]{gkz94}, then at a generic point (thus a smooth point) on this hypersurface , the corresponding tensor has a unique zero eigenvector \cite[Theorem 1.1.1.5]{gkz94}. Since the eigenvectors of a tensor $\mathcal T$ for the eigenvalue $\lambda$ are exactly the eigenvectors of the tensor $\mathcal T-\lambda\mathcal I$ for the eigenvalue zero, the result now follows from Proposition~\ref{prop:eigenvalue-dual}. 
\end{proof}

%%%%%%%%%%%%%%%%%%
\section{Final Remarks}\label{sec:final}
In this article, we studied the relationship between the algebraic multiplicity and the geometric multiplicity of an eigenvalue of a given tensor. Unlike the case for a matrix, the set of eigenvectors of an eigenvalue for a tensor does not have to be a linear subspace. In several cases, we show that the algebraic multiplicity is bounded below by a quantity determined by the geometric multiplicity and the order of the tensor (cf.\ \eqref{relation}), which reduces to the geometric multiplicity in the matrix counterpart. In general, the algebraic multiplicity should be bounded below by a quantity given by the
number of irreducible components of the eigenvariety and their dimensions (includes the geometric multiplicity), i.e., \eqref{general-relation}. 

The study of \eqref{relation} and \eqref{general-relation} would be interesting: the right hand sides of which should be invariants under the orthogonal linear group action, while the left hand side is not. One can also consider \eqref{general-relation} in the scheme level, for example, if there is an irreducible component $V$ of an eigenvariety being a double line as a scheme, then one should count this line twice in the summand. It would not be surprising that the inequality in \eqref{general-relation} becomes equality in the scheme level.

\subsection*{Acknowledgement} We thank Professor Bernd Sturmfels for suggesting us the nomenclature eigenvariety. This work was started during our visit to National Institute for Mathematical Sciences, South Korea, which was supported by the 2014 Thematic Program on Applied Algebraic Geometry; and continued during our visit to Simons Institute for the Theory of Computing, University of California at Berkeley, which was supported by the 2014 Program on Algorithms and Complexity in Algebraic Geometry. 
This work is also partially supported by National Science Foundation of China (Grant No. 11401428). 
%%%%%%%%%%%%%%%%%%%%%%%%%%%%%%%%%%%%%%%%%%%%%%%%%%%%%%%%%%%%%%%%%
%%%%%%%%%%%%%%%%%%%%%%%%%%%%%%%%%%%%%%%%%%%%%%%%%%%%%%%%%%%%
\bibliographystyle{model6-names}

\end{document}